\documentclass[reqno]{amsart}
\usepackage{hyperref}
\usepackage{amssymb}
\usepackage{enumerate}  
\usepackage{color}
\usepackage{dsfont}
\usepackage{graphicx}
\newcommand{\R}{{\mathbb R}}
\newcommand{\Z}{{\mathbb Z}}

\newcommand{\T}{{\mathbb T}}

\newcommand{\cA}{{\mathcal A}}
\newcommand{\cB}{{\mathcal B}}
\newcommand{\cD}{{\mathcal D}}
\newcommand{\cJ}{{\mathcal J}}
\newcommand{\cF}{{\mathcal F}}
\newcommand{\cL}{{\mathcal L}}
\newcommand{\p}{\partial}
\newcommand{\norm}[1]{\Vert #1 \Vert}
\newcommand{\vab}[1]{\vert #1 \vert}

\begin{document}
\title[\hfilneg EJDE\hfil [Fractional Schr\"{o}dinger Equation on cylinders]
{$\text{L}^2$ Solutions for Cubic NLS Equation with higher order Fractional Elliptic/Hyperbolic Operators on $\R\times\T$ and $\R^2$}

\author[A. J. Corcho and L. P. Mallqui \hfil EJDE\hfilneg]
{Ad\'an J. Corcho and Lindolfo P. Mallqui}

\address{Ad\'an J. Corcho \newline
Universidad de C\'ordoba - UCO,
Departamento de Matem\'aticas,
Campus de Rabanales.
14071, C\'ordoba, Spain.}
\email{a.corcho@uco.es}

\address{Lindolfo P. Mallqui \newline
Universidade Federal do Rio de Janeiro - UFRJ,
Instituto de Matem\'atica.
21941-909, Rio de Janeiro - RJ, Brazil}
\email{lindolfciencias@gmail.com}

\keywords{Elliptic/Hyperbolic cubic nonlinear Schr\"odinger equation; 
\hfill\break\indent Cauchy problem; Well-poseddness}

\begin{abstract}
In this work, we consider the Cauchy problem for the cubic Schr\"o\-dinger equation posed on cylinder $\R\times\T$ with fractional derivatives  $(-\p_y^2)^{\alpha},\, \alpha >0$, in the periodic direction. The spatial operator includes elliptic and hyperbolic regimes. We prove $L^2$ global well-posedness results  when  $\alpha \ge 1$ by proving a $L^4 - L^2$  Strichartz inequality for the linear equation, following the ideas in \cite{Takaoka-Tzvetkov}, where it was considered the elliptical case with $\alpha=1$. Further, these results remain valid on the Euclidean environment  $\R^2$, so well-posedness in $L^2$  are also achieved in this case. Our proof in the elliptic (hyperbolic) case does not work in the small directional dispersion case $0<\alpha <1$ ($0<\alpha \leq 1$), respectively. 
\end{abstract}

\maketitle
\numberwithin{equation}{section}
\newtheorem{theorem}{Theorem}[section]
\newtheorem{lemma}[theorem]{Lemma}
\newtheorem{remark}[theorem]{Remark}
\newtheorem{proposition}[theorem]{Proposition}
\allowdisplaybreaks

\section{Introduction}
We consider the Initial Value Problem (IVP) associated with the cubic non\-linear Schr\"odinger equation (3-NLS) on cylinder $\R\times \T$  with higher fractional derivatives  in the periodic direction. More precisely, we will study the IVP:
\begin{equation}\label{F-CNLS}
	\begin{cases}
		i\partial_tu + \cL^{\pm}_{\alpha}u = \varepsilon|u|^2u,&(x,y)\in\R\times\T,~~ t\in\R, \medskip \\
		u(0; x,y)=\phi(x,y),&(x,y)\in \R\times\T,
	\end{cases}
\end{equation}
where $\varepsilon\in \R^*$,  $u=u(t;x,y)$ is a complex-valued function on $\R\times\R\times \T$, with $\T=\R/2\pi\Z$, and $\cL^{\pm}_{\alpha}$ denotes the pseudo-differential operator acting in the space variables, given by  
\begin{equation}\label{CNLS-L-operator}
	\cL^{\pm}_{\alpha}:=\p^2_x \mp (-\p^2_y)^{\alpha},\quad \text{with}\; \alpha > 0, 
\end{equation}
defined in Fourier variables  by
\begin{equation}\label{CNLS-L-operator-symbol}
	\widehat{\cL^{\pm}_{\alpha}f}(\xi, n)= -(\xi^2 \pm |n|^{2\alpha})\widehat{f}^{x,y}(\xi, n),
\end{equation}
where $\widehat{f}^{x,y}$ is the Fourier transform of $f$ on the cylinder $\R \times \T$, that is,
\begin{equation}\label{FourierTransf}
	\cF[f](\xi, n)=\widehat{f}^{x,y}(\xi, n):=\frac{1}{2\pi}\int_{-\infty}^{\infty}\int_0^{2\pi}f(x,y)e^{-i(x\xi + ny)}dxdy.
\end{equation}
So, the linear propagator  $U^{\pm}_{\alpha}$ of \eqref{F-CNLS} is given by 
\begin{equation}\label{F-Group}
	U_{\alpha}^{\pm}(t)\phi=e^{it\,\cL^{\pm}_{\alpha}}\phi = \cF^{-1}\big(e^{-it(\xi^2 \pm |n|^{2\alpha})}\cF[\phi]\big).
\end{equation}

The cases $\varepsilon<0$ and $\varepsilon>0$ are known as the {\it focusing} and {\it defocusing} regimes, respectively. Also, note that in the case $\alpha =1$,  
$$
\cL^{\pm}_1=
\begin{cases}
	\p^2_x + \p^2_y =:\Delta \; \text{(Laplacian  operator) in the case (+)},\medskip\\
	\p^2_x - \p^2_y =:\Box \; \text{(wave  operator) in the case (-)}.
\end{cases}
$$

When $\alpha=1$, the elliptic case appears as a model in several physical problems (see for example the references \cite{Newell,Scott-Chu-McLaughlin,Zakharov-Shabat}). On the other hand, in the hyperbolic case it describes, for instance, the gravity waves on liquid surface and ion-cyclotron waves in plasma (see \cite{Crawford-Saffman-Yuen, Kuznetsov-Turitsyn, Sulem-Sulem}). Both cases (elliptic/hyperbolic) of the IVP \eqref{F-CNLS} with $\alpha >0$, posed on the Euclidean domain $\R^2$, are covered in the study carried out in \cite{DeBouard} about the existence of analytic solutions. In the case $0< \alpha < 1$, the elliptic operators $\cL^{+}_{\alpha}$ (fractional directional Laplacian) in \eqref{CNLS-L-operator} appear in the context of parabolic equations (see \cite{Barles-2012, Barles-2014} and references therein) as toy models to describe local diffusion occurring only in the $x$ direction, while non-local diffusion occurs in the $y$-direction. Moreover, the defocusing elliptic case of the IVP  \eqref{F-CNLS} with $\alpha= 1/2$ (called Half-wave-Schrödinger equation),  defined on the cylinder $\R\times \T$,  was considered in \cite{Xu} to establish a modified scattering theory between small solutions to this model and small solutions to the cubic Szeg\H{o} equation. In the same work, the author infers the existence of unbounded global solutions of \eqref{F-CNLS} in the anisotropic space $L^2_xH^s_y(\R\times \T)$ for every $s>1/2$. In \cite{Bahri-2020}, the authors considered the Half-wave-Schrödinger equation with a more general nonlinearity ($|u|^{p-1}u$), posed on the Euclidean domain $\R^2$, and they showed local well-posedness ($1 < p \le 5$) in anisotropic space $L^2_xH^s_y(\R^2)$ for $s>1/2$. Furthermore, they presented some results concerning the existence and orbital stability of solitary waves. Very recently, in \cite{Esfahani-2024}, the elliptic case of IVP  \eqref{F-CNLS}, posed on $\R^2$, is considered when $0 < \alpha < 1$ and with nonlinearity $|u|^{p-2}u$,\, $2 < p < 2\frac{(1+\alpha)}{1-\alpha}$. More precisely, conditions for the existence of blow-up solutions are investigated.
	
For  every $\alpha >0$, the IVP \eqref{F-CNLS}  (elliptic/hyperbolic) formally enjoys the mass conservation law:
	\begin{equation}\label{mass}
		M[u](t)=\displaystyle\int_{\R}\int_{\T}\vab{u(t;x,y)}^2dxdy= M[u](0),
	\end{equation}
	for all $t\in \R$.  Therefore, it would be interesting to establish a well-posedness theory in $L^2(\R\times \T)$ similarly to the case $\alpha =1$ (Laplacian and wave operators). The main goal of this paper is to answer this question.

\subsection{Case $\boldsymbol{\alpha =1}$ (some previous well-posedness results in isotropic Sobolev spaces)}

In the case of classical Laplacian/wave operator ($\alpha=1$)\label{Int-sub-1},
local and global well-posedness for time evolution of the flow associated to the IVP \eqref{F-CNLS}, with initial data $\phi$ belonging to the classical Sobolev spaces $H^s(\cD)$ on plane domains $\cD=\R^2,\;\T^2\; \text{or}\; \R\times \T$,
has been considered for many authors. For instance, we have the following results:
 
\begin{itemize}
\item $\cD=\R^2$: Well-posedness for in  $L^2(\R^2)$ is a consequence of the time decay estimates coming from dispersion and the Strichartz inequalities, which have the same form for the elliptic or hyperbolic linear Schr\"odinger equations. For the elliptic operator  $\cL^{+}_1$ see \cite{Cazenave-Weissler} and for the  hyperbolic case $\cL^{-}_1$  we refer to \cite{Ghidaglia-Saut-90}, where the authors deduced Strichartz inequalities for operators $e^{it\cL}$, where
$$\cL=\sum\limits_{1\le i, j \le n} a_{ij}\p_{x_i}\p_{x_j},~~ a_{ij}\in \R,$$
with non-degenerate quadratic form $A=\big(a_{ij} \big)$.  In particular, $L^2(\R^2)$ is the Sobolev critical regularity ($s=0$) for well-posedeness, which is suggested by the scale invariance of the model, that is, if $u$ is a solution of \eqref{F-CNLS} with $\alpha = 1$ and initial data $\phi(x,y)$, then 
$$u_{\lambda}(t;,x,y)= \ u(\lambda^2t; \lambda x, \lambda y),\; \lambda>0,$$
is the respective  solution of \eqref{F-CNLS} with initial data $\lambda \phi(\lambda x, \lambda y)$. 
	
\medskip 
\item $\cD=\T^2:$ Global well-posedness in $H^s(\T^2)$, $s>0$, has been proven in the elliptic case $U^+_1$ (see the works \cite{Bourgain-93, Bourgain-99}). In \cite{Wang} the hyperbolic case $U^-_1$  was considered, 
getting local well-posedness for $s>\frac12$ and ill-posedness when $s<\frac12$.
	
\medskip 
\item $\cD =\R\times \T$: The elliptical case  $\cL^{+}_1$ was treated  in \cite{Takaoka-Tzvetkov}. Specifically, the authors obtained global well-posedness for \textit{small data} in $L^2(\R\times \T)$
by proving  a $L^4 - L^2$  Strichartz inequality for the group $U^+_1(t)=e^{it\Delta}$ with  $(x,y)\in \R\times \T$. More precisely, they showed
 \begin{equation}\label{Elliptic-Strichartz}
 \|U_1^{+}(t)\phi\|_{L^4(I\times \R\times \T)}\le C_I\|\phi\|_{L^2(\R\times \T)},
 \end{equation}
where $I\subset \R$ is an interval containing $t=0$ and $C_I$ is a positive constant that depends only on $|I|$ (measure of $I$). Due the symmetry of the laplacian operator, the results are also valid 
in the space $L^2(\T\times \R)$. Similar Strichartz estimates were obtained in \cite{Herr-Tataru-Tzvetkov} for the energy critical nonlinear Schr\"odinger equation in partially periodic domains, 
for instance: $\R^m\times\T^{4-m}$ with $m=2,3$. 
\end{itemize}

In view of the above comments, a natural question is to figure out what happens for the IVP \eqref{F-CNLS} posed on domains $\R^2$, $\R\times \T$  or $\T\times \R$, with initial data in $L^2$,  for more general $\alpha>0$ as already pointed out before the beginning of this section. The notion of criticality given below tells us that we do not expect well-posedness in $L^2$ for IVP \eqref{F-CNLS} for $0< \alpha < 1$.

\subsection{Notion of criticality in isotropic Sobolev spaces for $\boldsymbol{\alpha > 0}$ }
The equation in \eqref{F-CNLS} on the domain $\R^2$ has the following scaling symmetric property: if $u$ is a solution to \eqref{F-CNLS}, $u_{\lambda}$ is also a solution to  \eqref{F-CNLS}, where 
\begin{equation}\label{alpha-scaling}
	u(t;x,y) \longmapsto u_{\lambda}(t;x,y):=\lambda u(\lambda^2t; \lambda x, \lambda^{1/\alpha}y),\; \lambda >0,
\end{equation}
which establishes a notion of criticality in Sobolev spaces  $H^s(\R^2)$ for the IVP  \eqref{F-CNLS}. More precisely, 

\begin{itemize}
	\item the index $s$ is called {\it critical} if \; $\|D^su_{\lambda}(0; \cdot, \cdot )\|_{L^2(\R^2)} \sim \|D^su(0; \cdot, \cdot )\|_{L^2(\R^2)}$,
	
	\medskip 
	\item the index $s$ is called {\it subcritical} if \; $\|D^su_{\lambda}(0; \cdot, \cdot )\|_{L^2(\R^2)} \to \infty$ as $\lambda \to  \infty$,
	
	\medskip 
	\item the index $s$ is called {\it supercritical} if \; $\|D^su_{\lambda}(0; \cdot, \cdot )\|_{L^2(\R^2)}  \to 0$ as $\lambda \to  \infty$.
\end{itemize}

\noindent
Computing $\|D^su_{\lambda}(0; \cdot, \cdot )\|_{L^2(\R^2)}$  for $\lambda>0$ we have 
\begin{equation}\label{critical-scaling}
	\|D^su_{\lambda}(0; \cdot, \cdot )\|^2_{L^2(\R^2)}= \lambda^{1-1/\alpha}\int_{\R^2}\big(\lambda^2 \xi_1^2 + \lambda^{2/\alpha}\xi_2^2\big)^s|\widehat{u}(0;\xi)|^2d\xi,
\end{equation}
with $\xi=(\xi_1, \xi_2)$. So  one gets
\begin{equation}\label{critical-scaling-a}
	\|D^su_{\lambda}(0; \cdot, \cdot )\|^2_{L^2(\R^2)}=
	\begin{cases}
		\lambda^{1-1/\alpha +2s/\alpha}p_{\alpha}(\lambda)& \text{if}\quad  0< \alpha < 1, \medskip \\
		\lambda^{2s}\|D^su(0; \cdot, \cdot )\|^2_{L^2(\R^2)}& \text{if}\quad  \alpha = 1, \medskip \\
		\lambda^{1-1/\alpha +2s}q_{\alpha}(\lambda)& \text{if}\quad  \alpha > 1,
	\end{cases}
\end{equation}
where
\begin{align*}
	&p_{\alpha}(\lambda)=\int_{\R^2}\big(\lambda^{2(1-\frac1\alpha)} \xi_1^2 + \xi_2^2\big)^s|\widehat{u}(0;\xi)|^2d\xi\\
	&q_{\alpha}(\lambda)=\int_{\R^2}\big(\xi_1^2 + \lambda^{2(\frac1\alpha -1)} \xi_2^2\big)^s|\widehat{u}(0;\xi)|^2d\xi.
\end{align*}
Hence, since 
$$\lim\limits_{\lambda\to +\infty}p_{\alpha}(\lambda)=\|D^s_yu(0; \cdot, \cdot )\|^2_{L^2(\R^2)}\, \text{and}\,  \lim\limits_{\lambda\to +\infty}q_{\alpha}(\lambda)=\|D^s_xu(0; \cdot, \cdot )\|^2_{L^2(\R^2)},$$
we conclude that
\begin{equation}\label{critical-regularity-alpha}
	s_c:=
	\begin{cases}
		\frac{1-\alpha}{2}>0& \text{is the {\it critical} regularity for}\; 0< \alpha < 1,\medskip \\
		\quad\; 0\quad &\text{is the {\it critical} regularity for}\;\, \alpha =1,\medskip \\
		\frac{1-\alpha}{2\alpha}<0& \text{is the {\it critical} regularity for}\; \alpha >1.
	\end{cases}
\end{equation}
Hence, it is not expected well-posedness in $L^2$ for $0<\alpha<1$. 

Furthermore, concerning well-posedness in Sobolev spaces to \eqref{F-CNLS} we note that some ill-posedness results (below $L^2$ regularity) for the one-dimensional cubic NLS on the line can be adapted to 
establish the same results for \eqref{F-CNLS} on the cylinder $\R\times \T$. Indeed, if we consider the following Cauchy problem
\begin{equation}\label{Ex-1}
	\begin{cases}
		i\p_tu+ \cL^{\pm}_{\alpha} u= \varepsilon|u|^2u,\qquad & (t;x,y)\in \R\times \R \times \T,  \medskip \\
		u(0; x,y)=\varphi(x),
	\end{cases}
\end{equation}
with $\varphi$ depending only on the $x$-variable and $\varphi\in H^s(\R)$, it follows that 
$\tilde{\varphi}(x, y):=\varphi(x)\in  H^s(\R\times\T)$ with
$$\|\tilde{\varphi}\|_{H^s(\R\times\T)}=\|\varphi\|_{H^s(\R)},$$
and  solutions of the IVP: 
\begin{equation}\label{Ex-2}
	\begin{cases}
		i\p_tw+ \p^2_xw=\varepsilon|w|^2w,\qquad & x\in \mathbb{R},\;\; t\in \R,\medskip \\
		w(x,0)=\varphi(x),
	\end{cases}
\end{equation}
are also solutions of \eqref{Ex-1}. Assuming the existence of local solutions, the IVP \eqref{Ex-2} is  ill-posed in the following situations:
\begin{enumerate}[(i)]
	\item $s\in (-1/2, 0)$ and  $\varepsilon< 0$. In this case, for any $\delta>0$  the uniform continuous of the flow-map
	\begin{equation}\label{flow-map-1d} 
		\Phi: u_0\in H^s(\R) \longmapsto w\in C([0,\delta];  H^s(\R))
	\end{equation}
	fails  (see \cite{KPV}).
	
	\medskip 
	\item $s\le -1/2$ and $\varepsilon \neq 0$. In this case, for any $\delta>0$  the  flow-map \eqref{flow-map-1d}
	is discontinuous everywhere in $H^s(\R)$ (norm inflation arguments, see \cite{Oh}).
\end{enumerate}
Hence, the statements (i) and (ii) imply similar ill-posedness results for negative Sobolev regularity ($s< 0$) to the IVP \eqref{Ex-1}  for all $\alpha>0$. 

In view of previous discussion  we study in this work well-posedness for IVP \eqref{Ex-1} with initial data in $L^2(\R\times\T)$ and we get global well-posedness under smallness assumption on data.  
This result remains valid in $L^2(\R^2)$, but unfortunately cannot be adapted to the case $L^2(\T\times\R)$.

\subsection{Main results}
Consider the IVP \eqref{F-CNLS}, with $\alpha \ge 1$ and initial data $\phi \in L^2(\R\times \T)$. In this work we show that it is possible to get a Strichartz estimate similar to \eqref{Elliptic-Strichartz} in the case of the group $U^{\pm}_{\alpha}$ defined in \eqref{F-Group} for $\alpha \ge 1$ in the elliptic case (+)  and 
for  $\alpha > 1$ in the hyperbolic case (-). More precisely, we will prove the following main result:

\begin{theorem}[Strichartz estimate on $I\times\R\times\T$]\label{Th-F-Strichartz}
	Let  $\alpha \ge 1$ and  $I\subset\R_t$ an interval containing $t=0$. Then, there exists a positive constant $C_{\alpha, I}$, depending only on $\alpha$ and the measure of $I$, such that 
	\begin{equation}\label{F-Strichartz}
		\norm{U^{\pm}_{\alpha}(t)\phi}_{L^4(I\times\R\times\T)}\le C_{\alpha, I}\norm{\phi}_{L^2(\R\times\T)},
	\end{equation}
	for any $\phi \in L^2(\R\times\T)$ with $\alpha \ge 1$ for  the case (+) and $\alpha  > 1$ for the case (-). 
	Moreover, there exists a positive constant $\widetilde{C}_{\alpha, I}$, depending only on $\alpha$ and the measure of $I$, such that 
	\begin{equation}\label{NH-Strichartz}
		\left\|\int_{0}^{t}U^{\pm}_{\alpha}(t-t')f(t';\cdot)dt'\right\|_{L^4(I\times\R\times\mathbb{T})}\leq \widetilde{C}_{\alpha, I}\left\| f\right\| _{L^{4/3}(I\times\R\times\mathbb{T})},
	\end{equation}
	for any $f\in L^{4/3}(I\times\R\times\mathbb{T})$.
\end{theorem}

\begin{remark}\label{Remark-extra-difficulty}
 We highlight that, in comparison with the case  $U^+_1(t)$ covered in \cite{Takaoka-Tzvetkov}, the study in the case  $\alpha \neq 1$ causes extra technical difficulty in the manipulation of symbol $\tau + \xi^2 - |n|^{2\alpha}$, since one cannot make use of the good algebraic  structure of quadratic polynomial in two variables corresponding to the symbol when $\alpha = 1$. Indeed,  to estimate the measure of the set $\mathcal{G}_K$ in the proof of crucial Lemma \ref{medida-produto-alfa} we had to introduce two appropriate auxiliary sets, which is not necessary in the case of  $U^+_1(t)$.
\end{remark}

As in \cite{Takaoka-Tzvetkov}, in the context of Cauchy problem for the cubic elliptic NLS, Theorem \ref{Th-F-Strichartz} combined with Picard iteration scheme
applied to the integral equation
\begin{equation}\label{int-equation}
	u(t)= U^{\pm}_{\alpha}(t)\phi -i\varepsilon\int_0^t U^{\pm}_{\alpha}(t-t')|u(t')|^2u(t')dt'
\end{equation}	
and the mass conservation law \eqref{mass} imply the following result:
\begin{theorem}[Well-posedness in $L^2$]\label{Th-Wp-L2}
	The  Cauchy problem \eqref{F-CNLS} is globally well-posed for sufficiently small data $\phi$ in $L^2(\R\times\T)$ with $\alpha \ge 1$ for  the case (+) and $\alpha  > 1$ for the case (-). 
\end{theorem}

\subsection{Final Comments} Finally, we point out some remarks. 

\begin{enumerate}[(I)]
	\item From the notion of criticality given in \eqref{critical-regularity-alpha} it is natural to expect  well-posedness results for the IVP \eqref{F-CNLS} for $s > \frac{1-\alpha}{2}>0$ when $0< \alpha < 1$.
	In fact, our proof of Theorem \ref{Th-F-Strichartz} does not work in this setting. See remarks \ref{remark-alpha-small-hyp} and \ref{remark-alpha-small-elip}.
	
	\medskip 
	\item The proof of Theorem \ref{Th-F-Strichartz} also fails for the group $U^{-}_1$, where the critical regularity suggested by the scaling is $L^2$. Thus, the hyperbolic case with $\alpha =1$  remains an interesting 
	open problem on the cylinder domain. At this point it is important to remember that in the hyperbolic case with  $\alpha =1$ on  $\R\times \R$, the critical regularity for well-posedness is $L^2$, but on $\T\times \T$ it was flagged 
	in \cite{Wang} that the optimal regularity must be $s=\frac12$.
	
	\medskip 
	\item The proof of Theorem \ref{Th-F-Strichartz} can also be performed in $\R^2$. Indeed, similar to the case $U^{+}_1$ treated in  \cite{Takaoka-Tzvetkov}, the Bourgain's method to obtain Strichartz inequalities 
	on cylinder for $U^{\pm}_{\alpha}$ also provides a proof of the Strichartz estimate with data  on $\R^2$ without using the time decay estimates coming from dispersion. Also, for subcritical nonlinearity  $|u|^{p-1}u\, (1< p <3)$ 
	instead the critical case $|u|^2u\, (p=3)$ globall well-posedness for any data in $L^2(\R\times \T)$ can be achieved in the same way as done in \cite{Takaoka-Tzvetkov} in the case  $U^{+}_1$.
	
	\medskip 
	\item Our approach cannot automatically adapt to the cylinder $\T\times \R$ since our strategy is based on the quadratic structure of the operator symbol with respect to the continuous propagation.
\end{enumerate}

\subsection{Notations and an elementary inequality}
Throughout the paper we will use the following notations: 
\begin{itemize}
	\item $|J|$  denotes the Lebesgue measure of a set $J\subset \R$,
	\item $\mathfrak{m}(\cdot)$ denotes the product measure of the one-dimensional Lebesgue and counting measure,
	\item $0< A(v)\lesssim B(v)$ means that there exists a positive constant $c$ such that $A(v)\leq cB(v)$,  for any $v$ varying on a certain set,
	\item $\lfloor x \rfloor$  denote the integer part of $x$.
\end{itemize}

\noindent 
Furthermore,  $X_{\alpha\pm}^{b,s}(\R\times\R\times\T)$ will denote the Bourgain space associated to the group $U^{\pm}_{\alpha}$, equipped with the norm 
\begin{equation}\label{H-BourgainNorm}
	\begin{split}
		\norm{f}_{X^{b,s}_{\alpha\pm}}^2
		&:= \norm{U^{\pm}_{\alpha}(-t)f}^2_{H^b_tH^s(\R\times\T)}\\
		&=\int_{\tau}\int_{\xi} \sum_{n\in \Z} \langle\, \vab{\xi}\!+\!\vab{n}\,\rangle^{2s}
		\langle\, \tau+ q_{\pm}(\xi, n)\,\rangle^{2b}\vab{\widehat{f}^{t,x,y}(\tau;\xi, n)}^2 d\tau d\xi,
	\end{split}
\end{equation}
where $\langle \cdot \rangle = 1 + |\cdot|$,
\begin{equation}\label{H-BourgainNorm-a}
	q_{\pm}(\xi, n):=\xi^2\pm|n|^{2\alpha}
\end{equation}
and $\widehat{f}^{t,x,y}(\tau;\xi, n)$ denotes the Fourier transform 
\begin{equation*}
  \widehat{f}^{t,x,y}(\tau; \xi,n)=\int_{-\infty}^{\infty}\int_{-\infty}^{\infty}\int_0^{2\pi}f(t;x,y)e^{-i(t\tau +x\xi + ny)}dtdxdy
\end{equation*}
with $(\tau; \xi, n)\in \R\times \R\times \T$.

The next inequality will be useful to establish some future estimates. 
\begin{lemma}\label{elem-inequality-lem}
 Let $\lambda$ be a positive number such that $0< \lambda <1$. Then, it holds that
	\begin{equation}\label{elem-inequality-lem-a}
		(a + b)^{\lambda}\le a^{\lambda} + b^{\lambda},
	\end{equation}
 for all  $a, b\ge 0.$
\end{lemma}

\noindent 
For the sake of completeness we give a proof for this ine\-quality.
	\begin{proof}
		The cases $a=0$ or  $b=0$ are obvious, so we consider $a, b$  strictly positive. Since $\lambda \in (0,1)$ we have 
		$$\frac{a}{a+b} < \Big(\frac{a}{a+b}\Big)^{\lambda} \quad \text{and}\quad  \frac{b}{a+b} < \Big(\frac{b}{a+b}\Big)^{\lambda},$$
		which imply
		\begin{equation}\label{elem-inequality-lem-b}
			1 <  \Big(\frac{a}{a+b}\Big)^{\lambda} + \Big(\frac{b}{a+b}\Big)^{\lambda},
		\end{equation}
		for all $a, b>0.$ The inequality \eqref{elem-inequality-lem-a} is an immediate consequence of \eqref{elem-inequality-lem-b}.
	\end{proof}

\section{Bilinear estimate in the hyperbolic case}
In this section, we present the proof of the key bilinear estimate,  which allows to get the $L^4$ Strichartz estimate for $U^{-}_{\alpha}$, 
with $\alpha >1$. To derive this bilinear estimate we need, as in the elliptic case treated in \cite{Takaoka-Tzvetkov}, to obtain uniform estimates of the measures 
for certain special sets (see \cite[Lemma 2.1]{Takaoka-Tzvetkov}). In our context the corresponding sets are unbounded and to estimate their measures is necessary to deal with series estimation. In particular, the convergence of such series occurs in the case $\alpha >1$ and fails in the case $0< \alpha \le 1$. 

We start by showing a similar result to that present in \cite[Lemma 2.1]{Takaoka-Tzvetkov}, which in our context reads as follows:

\begin{lemma}\label{medida-produto-alfa}
	Let $\alpha >1$, $\xi_0\in \R$,  $n_0\in\Z$ and $C>0$. Then for all $K\geq 1$, the set
	$$\mathcal{G}_K:=\Big \lbrace (\xi, n) \in\R\times\Z: C\leq \big|(\xi -\xi_0)^2  - \tfrac{1}{2}\big(|n|^{2\alpha} 
	+ |n-n_0|^{2\alpha}\big)\big|\leq C+K\Big\rbrace$$
	satisfies the estimate 
	\begin{equation}\label{medida-produto-alfa-a}
		\sup\limits_{(\xi_0, n_0, C)\in  \Lambda}\mathfrak{m}(\mathcal{G}_K) \lesssim_{\alpha} K, 
	\end{equation} 
	where $\Lambda:=  \R\times \Z\times \R^+$.
\end{lemma}

\begin{proof} 
	By translation invariance it suffices to consider the case $\xi_0=0$.  First we write
	\begin{equation}\label{set-GKC-decomposition-alfa}
		\mathcal{G}_K = \mathcal{G}_K^+\, \dot{\cup}\;  \mathcal{G}_K^-,
	\end{equation}
	where
	\begin{align}
		&\mathcal{G}_K^+  =\mathcal{G}_K \cap \Big \lbrace (\xi, n) \in\R\times \Z: \xi^2 > \tfrac{1}{2}\big(|n|^{2\alpha} 
		+ |n-n_0|^{2\alpha}\big)\Big\rbrace,\label{set-GKC-alfa+}\\
		& \mathcal{G}_K^- =\mathcal{G}_K \cap \Big \lbrace (\xi, n) \in\R\times \Z: \xi^2 \leq  \tfrac{1}{2}\big(|n|^{2\alpha} 
		+ |n-n_0|^{2\alpha}\big)\Big\rbrace \label{set-GKC-alfa-}.
	\end{align}
	Next we estimate the measures of the sets  $\mathcal{G}_K^+$ and $\mathcal{G}_K^-$.
	
	\medskip 
	\noindent 
	{\bf Estimate of  $\boldsymbol{\mathfrak{m}(\mathcal{G}_K^+)}$}. Using \eqref{set-GKC-alfa+} we write 
	
	\begin{equation}\label{proof-main-lemma-alfa-01}
		\mathcal{G}_K^+ = \Big\lbrace (\xi, n) \in\R\times \Z:  C\leq\xi^2 - \tfrac{1}{2}\big(|n|^{2\alpha} + |n-n_0|^{2\alpha}\big)\leq C+K \Big\rbrace.
	\end{equation}

	Let $l$ such that $l \geq 1$ and define
	\begin{equation}\label{proof-main-lemma-alfa-02}
		h(l):=\mathfrak{m}\big(\big\lbrace (\xi, n) \in\R\times \Z:\xi^2- \tfrac{1}{2}\big(|n|^{2\alpha} + |n-n_0|^{2\alpha}\big) \leq l \big \rbrace \big);
	\end{equation}	
	then
	\begin{equation}\label{proof-main-lemma-alfa-02a}
		\mathfrak{m}\big(\mathcal{G}^+_K\big)= h(C+K) - h(C).
	\end{equation}

	\begin{figure}[h!]
		\centering
		\includegraphics[scale=0.2]{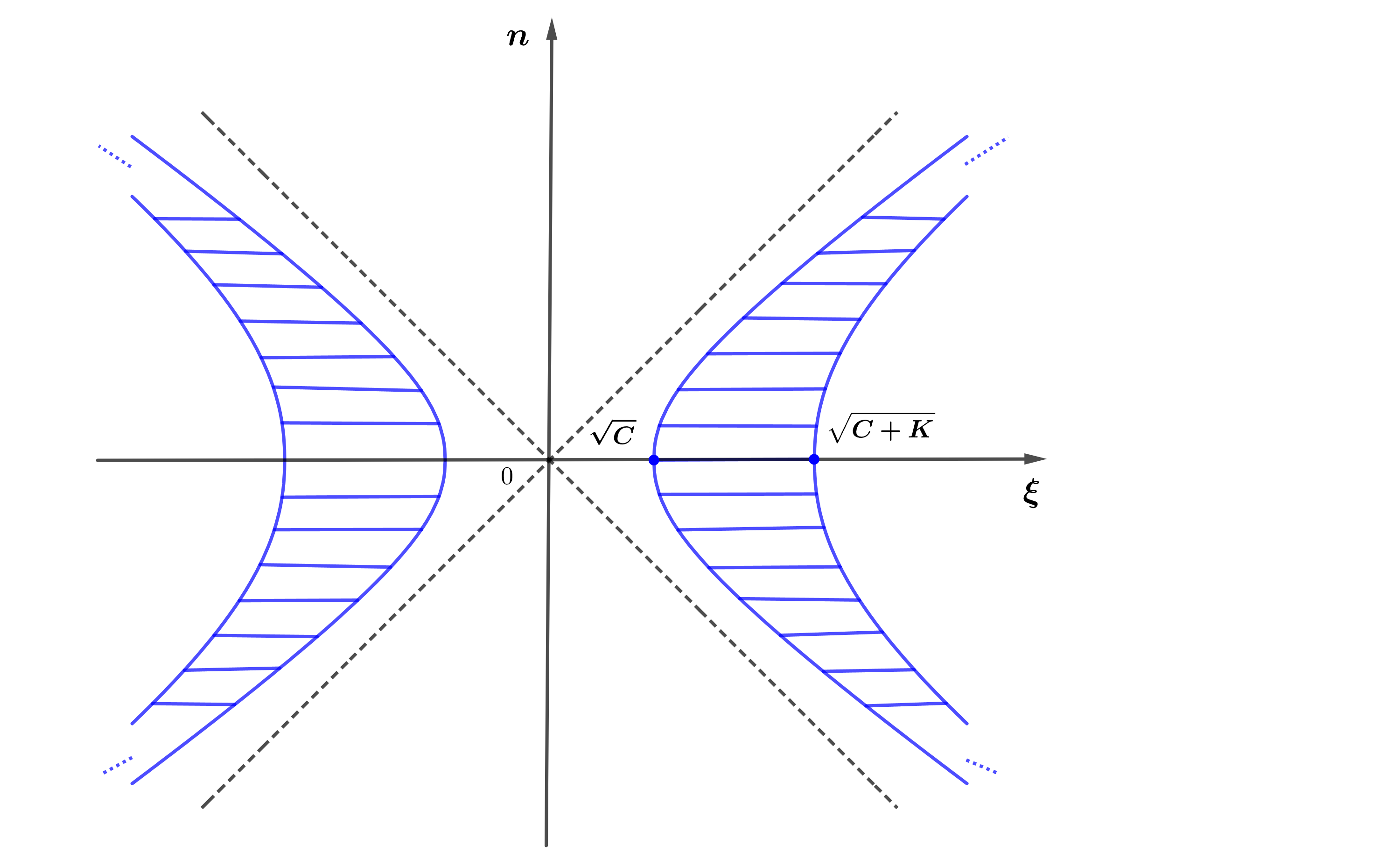}
		\caption{\small The dashed region represents the set $ \mathcal{G}^+_K$ with $\alpha = 1.3$ and $n_0=0$.}
		\label{fig-hyperbolic}
	\end{figure}

	Note that for a fixed $n\in \Z$ we have 
	$$\big|\big\lbrace \xi\in\R: \xi^2-\tfrac{1}{2}\big(|n|^{2\alpha} + |n-n_0|^{2\alpha})\leq l \big \rbrace \big|= 2\sqrt{\tfrac{1}{2}\big(|n|^{2\alpha} + |n-n_0|^{2\alpha}\big) +l};$$
	therefore
	\begin{multline}\label{proof-main-lemma-alfa-03}
		h(l)=2\displaystyle\sum_{n=0}^{\infty}\!\!\sqrt{\tfrac{1}{2}\big(|n|^{2\alpha} + |n-n_0|^{2\alpha}\big) +l}\\ + 2\displaystyle\sum_{n=1}^{\infty}\!\!\sqrt{\tfrac{1}{2}\big(|n|^{2\alpha} + |n+n_0|^{2\alpha}\big) +l}.
	\end{multline}
	Then, from\eqref{proof-main-lemma-alfa-02a} and \eqref{proof-main-lemma-alfa-03} it follows that 
	\begin{equation}\label{proof-main-lemma-alfa-03a}
		\mathfrak{m}\big(\mathcal{G}^+_K\big)= S_1^+ + S_2^+,
	\end{equation}
	where
	\begin{align*}
		& S_1^+= 2\displaystyle\sum_{n=0}^{\infty}\!\!\sqrt{\tfrac{1}{2}\big(|n|^{2\alpha} + |n-n_0|^{2\alpha}\big) +C+K} - 2\displaystyle\sum_{n=0}^{\infty}\!\!\sqrt{\tfrac{1}{2}\big(|n|^{2\alpha} + |n-n_0|^{2\alpha}\big)+C}\\
		& S_2^+=2\displaystyle\sum_{n=1}^{\infty}\!\!\sqrt{\tfrac{1}{2}\big(|n|^{2\alpha} + |n+n_0|^{2\alpha}\big) +C+K} -
		2\displaystyle\sum_{n=1}^{\infty}\!\!\sqrt{\tfrac{1}{2}\big(|n|^{2\alpha} + |n+n_0|^{2\alpha}\big) +C}.
	\end{align*}
	For $S_1^+$ we have
	$$S_1^+= 2\sum_{n=0}^{\infty}\frac{K}{\sqrt{\tfrac{1}{2}\big(|n|^{2\alpha} + |n-n_0|^{2\alpha}\big) +C+K}+\sqrt{\tfrac{1}{2}\big(|n|^{2\alpha} + |n-n_0|^{2\alpha}\big)+C}},$$
	then
	\begin{equation}\label{proof-main-lemma-alfa-03b}
		\begin{split}
			S_1^+& \le 2\sqrt{2}\sum_{n=0}^{\infty}\frac{K}{\sqrt{n^{2\alpha}+C+K}}\\
			&\le 2\sqrt{2}\Big( \frac{K}{\sqrt{C+K}} + K\sum_{n=1}^{\infty}\frac{1}{n^{\alpha}}\Big)\\
			&\lesssim_{\alpha}K, 
		\end{split}
	\end{equation}
	where in the last estimate it has been  used that $\alpha>1$ and $C+K\ge 1$. In a similar way one gets
	\begin{equation*}\label{proof-main-lemma-alfa-03c}
		S_2^+\lesssim_{\alpha}K.
	\end{equation*}
	Therefore, from \eqref{proof-main-lemma-alfa-03a} we have 
	\begin{equation}\label{proof-main-lemma-alfa-03d}
		\mathfrak{m}\big(\mathcal{G}^+_K\big)\lesssim_{\alpha}K.
	\end{equation}
	
	\medskip 
	\noindent
	{\bf Estimate of  $\boldsymbol{\mathfrak{m}(\mathcal{G}_K^-)}$}. This case is more delicate. We write
	\begin{equation}\label{proof-main-lemma-alfa-04}
		\mathcal{G}_K^- = \Big\lbrace (\xi, n) \in \R\times \Z: -(C+K)\leq \xi^2-\tfrac{1}{2}\big(|n|^{2\alpha} + |n-n_0|^{2\alpha}\big)\leq -C\Big\rbrace
	\end{equation}
	and observe that 
	\begin{equation}\label{proof-main-lemma-alfa-04a}
		\mathcal{G}_K^- \subset \mathcal{G}_{1,K}^{-} \cup 	\mathcal{G}_{2,K}^-,
	\end{equation}
	where 
	\begin{align*}
		&\mathcal{G}_{1,K}^{-} =  \Big\lbrace (\xi, n) \in \R\times \Z: |n-n_0|^{2\alpha}-(C+K)\leq \xi^2 \leq |n|^{2\alpha}-C\Big\rbrace,\\
		&\mathcal{G}_{2,K}^{-} =  \Big\lbrace (\xi, n) \in \R\times \Z: |n|^{2\alpha}-(C+K)\leq \xi^2 \leq |n-n_0|^{2\alpha}-C\Big\rbrace.
	\end{align*}
	Indeed, $\mathcal{G}_{1,K}^{-}$ contains the points of $\mathcal{G}_K^{-} $ with $|n-n_0|\le |n|$ and  $\mathcal{G}_{2,K}^{-}$  those that satisfy  $|n-n_0| > |n|$.
	
	Similar arguments to those used to estimate $\mathcal{G}_K^+$ give us 
	$$\mathfrak{m}\big(\mathcal{G}^-_{1,K}\big) =4\Big(\sum_{n=\lfloor C^{\frac{1}{2\alpha}}\rfloor+1}^{\infty}\!\!\!\sqrt{n^{2\alpha}-C}\,\,\, -\!\!\!\!\!\sum_{n-n_0 = \lfloor (C+K)^{\frac{1}{2\alpha}}\rfloor+1}^{\infty}\!\!\!\!\!\!\!\sqrt{(n-n_0)^{2\alpha}-C-K}\Big),$$
	which implies
	\begin{equation}\label{proof-main-lemma-alfa-05}
		\begin{split}
			\mathfrak{m}\big(\mathcal{G}^-_{1,K}\big)
			&=4\Big(\sum_{n=\lfloor C^{\frac{1}{2\alpha}}\rfloor+1}^{\infty}\!\!\!\sqrt{n^{2\alpha}-C}\,\,\, - \!\!\!\sum_{n= \lfloor (C+K)^{\frac{1}{2\alpha}}\rfloor+1}^{\infty}\!\!\!\!\!\sqrt{n^{2\alpha}-C-K}\Big)\\
			&:=4(S^-_1+S^-_2),
		\end{split}
	\end{equation}
	where
	\begin{align*}
		S^-_1&=\sum_{n=\lfloor C^{\frac{1}{2\alpha}}\rfloor+1}^{\lfloor (C+K)^{\frac{1}{2\alpha}}\rfloor}\sqrt{n^{2\alpha}-C},\\
		S^-_2&=\sum_{n=\lfloor(C+K)^{\frac{1}{2\alpha}}\rfloor+1}^{\infty}\big(\sqrt{n^{2\alpha}-C}-\sqrt{n^{2\alpha}-C-K}\big).
	\end{align*}

	To estimate $S_1^{-}$ we observe that
	\begin{equation}\label{proof-main-lemma-alfa-06}
		\begin{split}
			S^-_1&\leq \int_{C^{\frac{1}{2\alpha}}}^{\lfloor (C+K)^{\frac{1}{2\alpha}}\rfloor}\!\!\sqrt{z^{2\alpha}-C}\,dz 
			+ \sqrt{\lfloor (C+K)^{\frac{1}{2\alpha}}\rfloor^{2\alpha}-C}\\
			&\leq \underbrace{\int_{C^{\frac{1}{2\alpha}}}^{(C+K)^{\frac{1}{2\alpha}}}\!\!\sqrt{z^{2\alpha}-C}\,dz}_{\cJ_1} + \sqrt{K}. 
		\end{split}
	\end{equation}
	
	\noindent 
	To estimate the integral $\cJ_1$ we use that  $z \mapsto \sqrt{z^{2\alpha}-C}$ is an increasing function to get
		\begin{equation*}\label{proof-main-lemma-alfa-07}
			\cJ_1 \leq \sqrt{K}  \big((C+K)^{\frac{1}{2\alpha}}- C^{\frac{1}{2\alpha}}\big).	
		\end{equation*}
	Then, applying Lemma \ref{elem-inequality-lem} (remember that  $\alpha >1$) one gets 
		\begin{equation*}\label{proof-main-lemma-alfa-07a}
			\cJ_1 \leq \sqrt{K} K^{\frac{1}{2\alpha}}\le K,	
		\end{equation*}
	where in the last estimate it has been  used that $K\geq 1$ and $\alpha >1.$ Therefore, from \eqref{proof-main-lemma-alfa-06} and the estimate for $\cJ_1$ we conclude that  $S^{-}_1 \lesssim K.$

	Now we estimate $S_2^-$ as follows
	\begin{equation}\label{proof-main-lemma-alfa-10}
		\begin{split}
			S^-_2&=\sum_{n=\lfloor(C+K)^{\frac{1}{2\alpha}}\rfloor+1}^{\infty}\frac{K}{\sqrt{n^{2\alpha}-C}+\sqrt{n^{2\alpha}-C-K}}\\
			&\leq \sum_{n=\lfloor(C+K)^{\frac{1}{2\alpha}}\rfloor+1}^{\infty}\!\!\!\frac{K}{\sqrt{n^{2\alpha}-C}}\\
			&\leq \frac{K}{\sqrt{\big(\lfloor(C+K)^{\frac{1}{2\alpha}}\rfloor+1\big)^{2\alpha} - C}} + \underbrace{\int_{(C+K)^{\frac{1}{2\alpha}}}^{\infty}
				\frac{K}{\sqrt{z^{2\alpha}-C}}dz}_{\cJ_2}\\
			&\leq \sqrt{K} + \cJ_2.
		\end{split}
	\end{equation}
	Now we proceed to estimate the integral $\cJ_2$. Considering  the nonlinear change of  variables  $\rho^{2\alpha}=z^{2\alpha} -C$, and using that $\alpha >1$, $C>0$ and $K\ge1$, we have
		\begin{equation}\label{proof-main-lemma-alfa-10a}
			\begin{split}
				\cJ_2&=K\int_{K^{\frac{1}{2\alpha}}}^{\infty}\frac{\rho^{2\alpha-1}}{\rho^{\alpha}\big( \rho^{2\alpha} + C\big)^{1-\frac{1}{2\alpha}}}d\rho\\
				&\le K\int_{K^{\frac{1}{2\alpha}}}^{\infty}\rho^{-\alpha}d\rho\\
				&=\frac{1}{\alpha-1}KK^{\frac{1-\alpha}{2\alpha}}\\
				&\lesssim_{\alpha}K.
			\end{split}
		\end{equation}
Then, inserting \eqref{proof-main-lemma-alfa-10a} in \eqref{proof-main-lemma-alfa-10} one gets $S^-_2 \lesssim_{\alpha} K.$
	
Therefore, from \eqref{proof-main-lemma-alfa-05} and the estimates obtained for $S^-_1$ and $S^-_2$ we have 
	$$\mathfrak{m}\big(\mathcal{G}^-_{1,K}\big)\lesssim_{\alpha}K.$$ 
By the same way we have $\mathfrak{m}\big(\mathcal{G}^-_{2,K}\big)\lesssim_{\alpha}K.$ So, 
	$$\mathfrak{m}\big(\mathcal{G}^-_K\big) \le \mathfrak{m}\big(\mathcal{G}^-_{1,K}\big) + \mathfrak{m}\big(\mathcal{G}^-_{2,K}\big)  \lesssim_{\alpha}K,$$ 
which joint with \eqref{proof-main-lemma-alfa-03d} give us \eqref{medida-produto-alfa-a}. This completes  the proof of Lemma \ref{medida-produto-alfa}.
\end{proof}

\begin{remark}\label{remark-alpha-small-hyp} For the case $0< \alpha \leq 1$  the result stated in Lemma \ref{medida-produto-alfa} fails. Indeed, the series $S_1^+$ in \eqref{proof-main-lemma-alfa-03b} diverges for all
	$0< \alpha \le 1$. 
\end{remark}

\begin{proposition}[hyperbolic bilinear estimate]\label{Prop-bilinear-alpha}
	Let $u_1$ and $u_2$ be two functions on $L^2(\R\times\R\times\T)$ with the following support properties: 
	\begin{equation*}
		\text{supp}\,(\widehat{u_j}) \subseteq \mathcal{H}^-_{K_j},~~ j=1,2,
	\end{equation*}
	where
	\begin{equation}\label{Prop-bilinear-alpha-a}
		\mathcal{H}^-_{K_j}:=\Big\{(\tau;\xi,n):\tfrac1{2}K_j\le \vab{\tau+\xi^2-|n|^{2\alpha}}\leq 2K_j\Big\}.
	\end{equation}
	Then, we have the following inequality  
	\begin{equation}\label{Prop-bilinear-alpha-b}
		\norm{u_1u_2}_{L^2(\R\times\R\times\mathbb{T})}\lesssim (K_1K_2)^{\frac12}~\norm{u_1}_{L^2(\R\times\R\times\mathbb{T})}~\norm{u_2}_{L^2(\R\times\R\times\mathbb{T})}.
	\end{equation} 
\end{proposition}

\begin{proof}
	Set $(\tau_2;\xi_2, n_2): = (\tau-\tau_1;\xi-\xi_1,n-n_1)$. Using the Cauchy-Schwarz inequality and Plancherel's theorem, we have
	\begin{equation}\label{proof-Prop-bilinear-alpha-00}
		\begin{split}
			\norm{u_1u_2}_{L^2_{txy}}^2&\!=\!\int_{\tau}\!\int_{\xi} \sum_{n\in \Z} \Big|\int_{\tau_1}\!\int_{\xi_1}
			\sum_{n_1\in \Z}\!\!\mathds{1}
			_{\cA_{\tau\xi n}}\!(\tau_1;\xi_1,n_1) \cdot \widehat{u_1}^{t,x,y}(\tau_1;\xi_1,n_1)\\
			&\hspace{3.95cm}\cdot \widehat{u_2}^{t,x,y}(\tau_2;\xi_2,n_2)d\tau_1 d\xi_1 \Big|^2 d\tau d\xi\\
			&\lesssim\underset{\tau, \xi, n}{\mbox{sup}}\,\mathfrak{m}(\cA_{\tau\xi n}) \norm{u_1}^2_{L^2_{t x y}}\norm{u_2}^2_{L^2_{t x y}},
		\end{split}
	\end{equation}
with  $\cA_{\tau,\xi,n}$ defined as follows:
		\[(\tau_1;\xi_1,n_1)\in \cA_{\tau,\xi,n} \Longleftrightarrow 
		\begin{cases}
			(\tau_1;\xi_1,n_1)\in  \mbox{supp}\,(\widehat{u_1}^{t,x,y})\\
			\hspace{1.5cm}\text{and}\\
			(\tau_2;\xi_2, n_2)\in  \mbox{supp}\,(\widehat{u_2}^{t,x,y}).
		\end{cases}
		\]
Notice that if $(\tau_1;\xi_1, n_1) \in \cA_{\tau,\xi, n}$, then 
	\begin{equation*}\label{proof-Prop-bilinear-alpha-01}
		\tau_1 \in J_1:= \big[a_1 -2K_1,\, a_1+2K_1\big]  \quad \text{and} \quad \tau_1 \in J_2:= \big[a_2 -2K_2,\, a_2 + 2K_2\big],
	\end{equation*}
	where $a_1 = |n_1|^{2\alpha} -\xi_1^2$ and $a_2 = \tau +(\xi-\xi_1)^2 - |n-n_1|^{2\alpha}$. Hence
	\begin{equation}\label{proof-Prop-bilinear-alpha-02}
		\tau_1 \in J_1\cap J_2, \;\text{with}\; |J_1\cap J_2| \le 4\min\{K_1, K_2\}.
	\end{equation}
	
	On the other hand,  for all $(\tau_1;\xi_1,n_1)\in \cA_{\tau,\xi,n}$ we can use the triangle inequality to eliminate $\tau_1$ to get 
	\begin{equation}\label{proof-Prop-bilinear-alpha-03}
		\begin{split}
			\Big|(\xi_1-&\tfrac{\xi}{2})^2  - \tfrac{1}{2}\big(|n_1|^{2\alpha} + |n-n_1|^{2\alpha}\big) + \tfrac{\tau}{2} +\tfrac{\xi^2}{4}\Big|\\
			& = \frac{1}{2}\big|\tau + \xi_1^2-|n_1|^{2\alpha} + (\xi-\xi_1)^2 - |n-n_1|^{2\alpha}\big|\\
			&\leq \tfrac{1}{2}\Big(\big|\tau_1+ \xi_1^2-|n_1|^{2\alpha}\big| +  \big|\tau -\tau_1 + (\xi-\xi_1)^2 - |n-n_1|^{2\alpha} \big| \Big)\\
			&< K_1 + K_2.
		\end{split}
	\end{equation}
	So, 
	\begin{equation}\label{proof-Prop-bilinear-alpha-04}
		\text{if}\; (\tau_1;\xi_1, n_1) \in \cA_{\tau,\xi, n},\; \text{then}\; (\xi_1, n_1) \in \cB_{\tau,\xi, n}\footnote{This is the reason that led us to consider $\mathcal{G}_K$ slightly different from the set that should naturally replace the used in \cite[Lemma 2.1]{Takaoka-Tzvetkov}.}, 
	\end{equation}
	where
	$$\cB_{\tau,\xi, n}:=\Big\{(\xi_1, n_1)\in\R\times \Z: \big|(\xi_1-\tfrac{\xi}{2})^2  - \tfrac{1}{2}\big(|n_1|^{2\alpha} 
	+ |n_1-n|^{2\alpha}\big) + \tfrac{\tau}{2} + \tfrac{\xi^2}{4}\big| \le K_1 + K_2\Big\}.$$
	
	\noindent 
	Now, from triangle inequality, we note that 
	$$\cB_{\tau,\xi, n} \subset \Big\lbrace (\xi_1, n_1) \in\R\times \Z: C\leq \big|(\xi_1-\tfrac{\xi}{2})^2 - \tfrac{1}{2}\big(|n_1|^{2\alpha} + |n_1-n|^{2\alpha}\big)\big| \leq C +2(K_1+K_2)\Big\rbrace$$
	with $C=\big|\tfrac{\tau}{2}+\tfrac{\xi^2}{4}\big| - (K_1 +K_2).$ Hence, if $C>0$ we use Lemma \ref{medida-produto-alfa} (with $\xi_0=\tfrac{\xi}{2}$ and $n_0=n$) to get 
	$$\frak{m}(\cB_{\tau,\xi, n})\lesssim K_1+K_2.$$
	Otherwise, if $C<0$ we have 
	$$\frak{m}(\cB_{\tau,\xi, n})  \leq \frak{m} \Big(\Big\{(\xi_1, n_1): \big|(\xi_1-\tfrac{\xi}{2})^2  - \tfrac{1}{2}\big(|n_1|^{2\alpha} + |n_1-n|^{2\alpha}\big)\big| \le 2(K_1 + K_2)\Big\} \Big)$$
	and consequently
	$$\frak{m}(\cB_{\tau,\xi, n})\le \lim\limits_{\varepsilon \searrow 0}\frak{m} \Big(\Big\{(\xi_1, n_1): \varepsilon< \big|(\xi_1-\tfrac{\xi}{2})^2  - \tfrac{1}{2}\big(|n_1|^{2\alpha} 
	+ |n_1-n|^{2\alpha}\big)\big| \le 2(K_1 + K_2)\Big\}\Big),$$
	so using again Lemma \ref{medida-produto-alfa} it follows that
	\begin{equation}\label{proof-Prop-bilinear-alpha-05}
		\frak{m}(\cB_{\tau,\xi, n})\lesssim K_1 + K_2.
	\end{equation} 
	
	Finally, collecting the information in \eqref{proof-Prop-bilinear-alpha-02},  \eqref{proof-Prop-bilinear-alpha-04} and \eqref{proof-Prop-bilinear-alpha-05} one gets 
	$$\mathfrak{m}(\cA_{\tau,\xi, n})\le |J_1\cap J_2|\cdot \mathfrak{m}(\cB_{\tau,\xi, n}) \lesssim \min\{K_1,K_2\}\max\{K_1,K_2\}$$
	and inserting this in \eqref{proof-Prop-bilinear-alpha-00} we obtain 
	\begin{equation*}
		\norm{u_1u_2}_{L^2_{txy}}^2\lesssim K_1K_2\norm{u_1}^2_{L^2_{t x y}}\norm{u_2}^2_{L^2_{t x y}}.
	\end{equation*}
	Then, the result is proved. 
\end{proof}

\section{Bilinear estimate in the elliptic case} Now we prove the corresponding bilinear estimate for the elliptic symbol of the operator.  

\begin{lemma}\label{medida-produto-alfa-elip}
	Let $\alpha \ge 1$, $\xi_0\in \R$,  $n_0\in\Z$ and $C\ge 1$. Then for all $K\geq 1$, the set
	$$\widetilde{\mathcal{G}}_K:=\Big \lbrace (\xi, n) \in\R\times\Z: C\leq (\xi -\xi_0)^2  + \tfrac{1}{2}\big(|n|^{2\alpha} 
	+ |n-n_0|^{2\alpha}\big)\leq C+K\Big\rbrace$$
	satisfies the estimate 
	\begin{equation}\label{medida-produto-alfa-elip-a}
		\sup\limits_{(\xi_0, n_0, C)\in  \Lambda}\mathfrak{m}(\widetilde{\mathcal{G}}_K) \lesssim_{\alpha} K, 
	\end{equation} 
	where $\Lambda:=  \R\times \Z\times \R^+$.
\end{lemma}

\begin{proof}
	As in Lemma \ref{medida-produto-alfa} it suffices to consider the case $\xi_0=0$. For all $K\ge 1$ notice that 
	\begin{equation}\label{proof-main-lemma-alfa-elip-a}
		\widetilde{\mathcal{G}}_K \subset \widetilde{\mathcal{G}}_{1,K} \cup  \widetilde{\mathcal{G}}_{2,K},
	\end{equation}
	where 
	\begin{align*}
		&\widetilde{\mathcal{G}}_{1,K} =  \Big\lbrace (\xi, n) \in \R\times \Z: C-|n|^{2\alpha}\leq \xi^2 \leq C+K -|n-n_0|^{2\alpha}\Big\rbrace,\\
		&\widetilde{\mathcal{G}}_{2,K} =  \Big\lbrace (\xi, n) \in \R\times \Z: C- |n-n_0|^{2\alpha}\leq \xi^2 \leq C+K -|n|^{2\alpha}\Big\rbrace.
	\end{align*}
	Indeed, $\widetilde{\mathcal{G}}_{1,K}$ contains the points of $\widetilde{\mathcal{G}}_K $ with $|n-n_0|\le |n|$ and  $\widetilde{\mathcal{G}}_{2,K}$  those that satisfy  $|n-n_0| > |n|$.
	
	Similar analysis as performed in the hyperbolic case shows that  
	\begin{equation}\label{proof-main-lemma-alfa-elip-01}
		\begin{split}
			\frak{m}(\widetilde{\mathcal{G}}_{1,K})&= 2\sum\limits_{|n-n_0| = 0}^{\lfloor (C+K)^{\frac1{2\alpha}}\rfloor}\sqrt{C+K - |n-n_0|^{2\alpha}}-2\sum\limits_{|n| = 0}^{\lfloor C^{\frac1{2\alpha}}\rfloor}\sqrt{C - |n|^{2\alpha}}\\
			&=2\sum\limits_{|n|= 0}^{\lfloor (C+K)^{\frac1{2\alpha}}\rfloor}\sqrt{C+K - |n|^{2\alpha}} -2\sum\limits_{|n| = 0}^{\lfloor C^{\frac1{2\alpha}}\rfloor}\sqrt{C - |n|^{2\alpha}}\\
			&:=\widetilde{S}_1 + \widetilde{S}_2,
		\end{split}
	\end{equation}
	where 
	\begin{align*}
		\widetilde{S}_1&=2\sum\limits_{|n| = 0}^{\lfloor C^{\frac1{2\alpha}}\rfloor}\Big(\sqrt{C+K - |n|^{2\alpha}} - \sqrt{C - |n|^{2\alpha}}\Big),\\
		\widetilde{S}_2&=2\sum\limits_{|n|=\lfloor C^{\frac1{2\alpha}}\rfloor +1 }^{\lfloor (C+K)^{\frac1{2\alpha}}\rfloor}\sqrt{C+K - |n|^{2\alpha}}.
	\end{align*}
	
	In order to estimate $	\widetilde{S}_1$ we use that 
	\begin{equation}\label{proof-main-lemma-alfa-elip-02}
		\begin{split}
			\widetilde{S}_1 &=  2\sum\limits_{|n| = 0}^{\lfloor C^{\frac1{2\alpha}}\rfloor}\frac{K}{\Big(\sqrt{C+K - |n|^{2\alpha}} + \sqrt{C - |n|^{2\alpha}}\Big)}\\
			&\le 2\sum\limits_{|n| = 0}^{\lfloor C^{\frac1{2\alpha}}\rfloor-1}\frac{K}{\sqrt{C - |n|^{2\alpha}}} +  \frac{2K}{\sqrt{C +K - \lfloor C^{\frac1{2\alpha}}\rfloor^{2\alpha}}}\\
			&\le 4 K \int_0^{C^{\frac1{2\alpha}}}\frac{dz}{\sqrt{C-z^{2\alpha}}} + 2\sqrt{K}.
		\end{split}
	\end{equation}
	Making now the change of variables  $z= C^{\frac{1}{2\alpha}}\rho^{\frac{1}{\alpha}}$ with $\alpha \ge 1$ we have from \eqref{proof-main-lemma-alfa-elip-02}
	\begin{equation}\label{proof-main-lemma-alfa-elip-03}
		\begin{split}
			\widetilde{S}_1 &\le  \frac{4K}{\alpha C^{\frac{\alpha -1}{2\alpha}}} \int_0^1\frac{d\rho}{\rho^{1-\frac{1}{\alpha}}\sqrt{1-\rho^2}} + 2\sqrt{K}\\
			&\lesssim  \frac{4K}{\alpha C^{\frac{\alpha -1}{2\alpha}}} \Big( \int_0^1\frac{d\rho}{\rho^{1-\frac{1}{\alpha}}} +  \int_0^1\frac{d\rho}{\sqrt{1-\rho^2}} \Big) + 2\sqrt{K}\\
			&\lesssim_{\alpha}K,
		\end{split}
	\end{equation}
where it has been  used that $C\ge 1$ and $\alpha\ge 1$. 
	
Now we proceed to estimate $\widetilde{S}_2$. First note that 
	\begin{equation}\label{proof-main-lemma-alfa-elip-04}
		\begin{split}
			\widetilde{S}_2& = 4\sum\limits_{n=\lfloor C^{\frac1{2\alpha}}\rfloor +1 }^{\lfloor (C+K)^{\frac1{2\alpha}}\rfloor}\sqrt{C+K - n^{2\alpha}}\\
			&\le  4\sqrt{K} + \int_{\lfloor C^{\frac1{2\alpha}}\rfloor +1}^{\lfloor (C+K)^{\frac1{2\alpha}}\rfloor}\sqrt{C+K - z^{2\alpha}}dz\\
			&\le  4\sqrt{K} + \int_{C^{\frac1{2\alpha}}}^{(C+K)^{\frac1{2\alpha}}}\sqrt{C+K - z^{2\alpha}}dz\\
			&\le  4\sqrt{K} +\sqrt{K}\big((C+K)^{\frac1{2\alpha}} -  C^{\frac1{2\alpha}}\big). 	
		\end{split}
	\end{equation}	
On the other hand, Lemma \ref{elem-inequality-lem-a} gives us that 
		\begin{equation}\label{proof-main-lemma-alfa-elip-05}
			(C+K)^{\frac1{2\alpha}} -  C^{\frac1{2\alpha}}\leq 	K^{\frac1{2\alpha}}
		\end{equation}
for all $C\ge 1$ and $\alpha\ge 1$. Combining \eqref{proof-main-lemma-alfa-elip-04} and \eqref{proof-main-lemma-alfa-elip-05} we have
	$$\widetilde{S}_2 \lesssim K,$$
then the proof is finished. 
\end{proof}

\begin{remark}\label{remark-alpha-small-elip} For the case $0< \alpha < 1$  the result stated in Lemma \ref{medida-produto-alfa-elip} fails. Indeed, it is not difficult to see that the
	$\widetilde{S}_1$ verifies  \eqref{proof-main-lemma-alfa-elip-03} in the following way 
	$$\widetilde{S}_1 \sim\Big( \frac{1}{C^{\frac{\alpha -1}{2\alpha}}}  +1\Big)K,$$
	where $\frac{1}{C^{\frac{\alpha -1}{2\alpha}}} \to + \infty$ as $C\to + \infty$ whenever  $0< \alpha < 1$. 
\end{remark}

In the same way as in the hyperbolic case, Lemma \ref{medida-produto-alfa-elip} implies the following result.

\begin{proposition}[elliptic bilinear estimate]\label{Prop-bilinear-alpha-elip}
	Let $u_1$ and $u_2$ be two functions on $L^2(\R\times\R\times\T)$ with the following support properties: 
	\begin{equation*}
		\text{supp}\,(\widehat{u_j}) \subseteq \mathcal{H}^+_{K_j},~~ j=1,2,
	\end{equation*}
	where
	\begin{equation}\label{Prop-bilinear-alpha-elip-a}
		\mathcal{H}^+_{K_j}:=\Big\{(\tau;\xi,n):\tfrac1{2}K_j\le \vab{\tau+\xi^2+|n|^{2\alpha}}\leq 2K_j\Big\}.
	\end{equation}
	Then we have the following inequality  
	\begin{equation}\label{Prop-bilinear-alpha-elip-b}
		\norm{u_1u_2}_{L^2(\R\times\R\times\mathbb{T})}\lesssim (K_1K_2)^{\frac12}~\norm{u_1}_{L^2(\R\times\R\times\mathbb{T})}~\norm{u_2}_{L^2(\R\times\R\times\T)}.
	\end{equation} 
\end{proposition}

\section{Strichartz Inequality on $\R\times \T$}

This section is devoted to the proof of Theorem ~\ref{Th-F-Strichartz}, which follows the same lines as previous proofs of related results in \cite{Takaoka-Tzvetkov} and we reproduce a sketch of it for the sake of completeness.  

\subsection{Proof of Theorem ~\ref{Th-F-Strichartz}--\eqref{F-Strichartz}} The first step is the proof of the following $L^4$ Strichartz estimate  in the 
Bourgain space $ X_{\alpha\pm}^{b,0}$.

\medskip
\noindent{\bf Step 1}. Let $\alpha\ge 1$ in the case (+), $\alpha > 1$ in the case (-) and $b>1/2$. Then we have
\begin{equation}\label{Prop-L4-Xb0-a}
	\norm{u}_{L^4(\R\times\R\times\mathbb{T})}\lesssim\norm{u}_{X_{\alpha\pm}^{b,0}(\R\times\R\times\mathbb{T})}
\end{equation}
for any $u\in X_{\alpha\pm}^{b,0}(\R\times\R\times\mathbb{T})$.

\medskip
\noindent
{\it Proof of Step 1}. Let $u$ a smooth function and consider the dyadic decomposition 
\begin{equation*}
	u(t; x, y)=\sum_{k=0}^{\infty}u_{2^k}(t; x, y), 
\end{equation*}
with $\text{supp}(\widehat{u_{2^k}}^{t,x,y}) \in \mathcal{H}^{\pm}_{2^k}$ defined in \eqref{Prop-bilinear-alpha-a}. Using Proposition \ref{Prop-bilinear-alpha},  Proposition \ref{Prop-bilinear-alpha-elip} and that $b>1/2$ one gets
\begin{equation}\label{Step 1-a}
	\begin{split}
		&\norm{u}^2_{L^4(\R\times\R\times\mathbb{T})}=\norm{uu}_{L^2(\R\times\R\times\mathbb{T})}\\ 
		&\qquad \leq \sum_{k_1=0}^{\infty}\sum_{k_2=0}^{\infty}\norm{u_{2^{k_1}}u_{2^{k_2}}}_{L^2(\R\times\R\times\mathbb{T})}\\ 
		&\qquad \lesssim \sum_{k_1=0}^{\infty}\sum_{k_2=0}^{\infty} (2^{k_1} 2^{k_2})^{1/2}\norm{\widehat{u_{2^{k_1}}}^{t,x,y}}_{L^2(\R\times\R\times\mathbb{T})}
		\norm{\widehat{u_{2^{k_2}}^{t,x,y}}}_{L^2(\R\times\R\times\mathbb{T})}.
	\end{split}
\end{equation}

\noindent 
Since $\text{supp}(\widehat{u_{2^k}}^{t,x,y}) \in \mathcal{H}^{\pm}_{2^k}$ we  have $\langle \tau + q_{\pm}(\xi, n)\rangle \ge |\tau + q_{\pm}(\xi, n)| \ge 2^{k-1}$.
Then, it follows that 
	\begin{equation}\label{Step 1-b}
		\begin{split}
			&\norm{\widehat{u_{2^{k_j}}}^{t,x,y}}_{L^2(\R\times\R\times\mathbb{T})}=\bigg(\int_{\tau}\int_{\xi}\sum\limits_{n\in \Z}
			\big|\widehat{u_{2^{k_j}}}^{t,x,y}(\tau;\xi, n)\big|^2d\tau d\xi\bigg)^{1/2}\\
			&\qquad \qquad \le \bigg(\int_{\tau}\int_{\xi}\sum\limits_{n\in \Z}\frac{\langle \tau + q_{\pm}(\xi, n)\rangle^{2b}}{2^{(k_j-1)2b}}
			\big|\widehat{u_{2^{k_j}}}^{t,x,y}(\tau;\xi, n)\big|^2d \tau d\xi\bigg)^{1/2}\\
			&\qquad \qquad \lesssim2^{-bk_j}\norm{u}_{X_{\alpha\pm}^{b,0}},
		\end{split}
   \end{equation}
for  $j=1,2$. So, inserting \eqref{Step 1-b} in \eqref{Step 1-a} and using that $b>1/2$, we have
\begin{equation}\label{Step 1-c}
	\begin{split}
		\norm{u}^2_{L^4(\R\times\R\times\mathbb{T})}
		&\lesssim \sum_{k_1=0}^{\infty}(2^{k_1})^{1/2-b}\sum_{k_2=0}^{\infty}(2^{k_2})^{1/2-b}\norm{u}^2_{X_{\alpha^{\pm}}^{b,0}}\\
		&\lesssim \norm{u}^2_{X_{\alpha\pm}^{b,0}},
	\end{split}
\end{equation}
as claimed in \eqref{Prop-L4-Xb0-a}.

\medskip 
\noindent{\bf Step 2}. Let $\delta >0$, $I=[-\delta, \delta]$ and $b>1/2$. Then, 
\begin{equation}\label{Prop-L4-Xb0-b}
	\norm{U^{\pm}_{\alpha}(t)\phi}_{L^4(I\times\R\times\T)}\lesssim (\delta^{1/2}+\delta^{1/2-b})\norm{\phi}_{L^2(\R\times\T)},
\end{equation}
for any $\phi \in  L^2(\R\times\T)$.

\medskip
\noindent
{\it Proof of Step 2}. Let $\psi\in C^\infty_0$ be a cut-off function such that $\mbox{supp~}(\psi) \subset(-2,2)$ and $\psi(t)\equiv 1$ on $[-1,1]$ and define  
$\psi_{\delta}(t):=\psi(\frac{t}{\delta})$. For $b>0$, one gets
$$\norm{\psi_{\delta}}_{H^b_t}\lesssim \delta^{1/2}\norm{\psi}_{L^2_t}+\delta^{1/2-b}\norm{\psi}_{\dot{H}^b_t}.$$
\noindent 
Now, applying \eqref{Prop-L4-Xb0-a} (here we need $b>1/2$), we have
\begin{equation*}
	\begin{split}
		\norm{U^{\pm}_{\alpha}(t)\phi}_{L^4(I\times\R\times\T)}&\lesssim \norm{\psi_{\delta}(t)U^{\pm}_{\alpha}(t)\phi}_{L^4(\R\times\R\times\T)}\\
		&\lesssim\norm{\psi_{\delta}(t)U^{\pm}_{\alpha}(t)\phi}_{X^{b,0}_{\alpha^{\pm}}(\R\times\R\times\mathbb{T})}\\
		&= \norm{\psi_{\delta}}_{H^b_t}\norm{\phi}_{L^2(\R\times\T)}\\
		&\lesssim (\delta^{1/2}+\delta^{1/2-b})\norm{\phi}_{L^2(\R\times\T)},
	\end{split}
\end{equation*}
so \eqref{Prop-L4-Xb0-b} is proved. 

Finally, the estimate \eqref{Prop-L4-Xb0-b}  can  be extended to an arbitrary interval $I$ in the same way as in \cite{Takaoka-Tzvetkov}. This completes the proof of Theorem ~\ref{Th-F-Strichartz}--\eqref{F-Strichartz}.

\subsection{Proof of Theorem ~\ref{Th-F-Strichartz}--\eqref{NH-Strichartz}} Consider the linear operator 
\[A: L^2(\R\times\mathbb{T}) \longrightarrow L^4(I\times\R\times\mathbb{T}),\]
defined by $A\phi:= U^{\pm}_{\alpha}(t)\phi$. Due to the estimate \eqref{F-Strichartz} we observe that $A$ is bounded. Let us denote by $A^*$ the adjoint operator of $A$, then
\begin{equation*}
	A^*(f)=\displaystyle\int_{I}U^{\pm}_{\alpha}(-t')f(t';\cdot)dt',
\end{equation*}
which is bounded from  $L^{4/3}( I\times\R\times\mathbb{T})$ to $L^2(\R\times\mathbb{T})$, i.e.,
\begin{equation*}
	\left\|\int_IU^{\pm}_{\alpha}(-t')f(t';\cdot)dt'\right\| _{L^2(\R\times\mathbb{T})}\leq C^*_I\left\| f\right\| _{L^{4/3}(I\times\R\times\mathbb{T})} 
\end{equation*}
for some constant $C^*_I$, depending only on the measure of $I$. Therefore,
\begin{equation*}
	AA^*=\displaystyle\int_{I}U^{\pm}_{\alpha}(t-t')f(t';\cdot)dt',
\end{equation*}
is bounded from $L^{4/3}( I\times\R\times\mathbb{T})$ to  $L^4(I\times\R\times\mathbb{T})$, satisfying 
\begin{equation*}
	\left\|\int_IU^{\pm}_{\alpha}(t-t')f(t';\cdot)dt'\right\| _{L^4(I\times\R\times\mathbb{T})}\leq C_IC^*_I\left\| f\right\| _{L^{4/3}(I\times\R\times\mathbb{T})}.
\end{equation*}
Finally, to complete the proof of  \eqref{NH-Strichartz} is used the arguments in Lemma 3.1 of \cite{Smith-Sogge}.

\subsection*{Acknowledgments}
This  research was partially supported by the Coordena\c{c}\~ao de Aperfei\c{c}oamento de Pessoal de N\'ivel Superior-Brasil (CAPES)-Finance Code 001. The first author was financed  by  CNPq/Brazil grant no. 307616/2020-7 
and ``Beatriz Galindo" research position at the University of C\'ordoba/Spain.

\end{document}